\documentclass[11pt]{amsart}
\usepackage{preamble,enumerate}

\def\Z{{\mathbb Z}}

\title{Farrell-Jones conjecture for free-by-cyclic groups}%

\author{Mladen Bestvina}
\address{\tt M.\ Bestvina, Department of Mathematics, University of Utah, 155 S.\ 1400 E.\,
  \newline Salt Lake City, UT 84112, U.S.A.
  \newline http://www.math.utah.edu/$\sim$bestvina/} %
\email{\tt bestvina@math.utah.edu}
\author{Koji Fujiwara}
\address{\tt K.\ Fujiwara, Department of Mathematics, Kyoto University. Kyoto 606-8502, Japan.
  \newline http://www.math.kyoto-u.ac.jp/$\sim$kfujiwara/} %
\email{\tt kfujiwara@math.kyoto-u.ac.jp}
\author{Derrick Wigglesworth} %
\address{\tt D.\ Wigglesworth, Department of Mathematical Sciences, University
  of Arkansas, 309 SCEN, Fayetteville, AR 72703, U.S.A.
  \newline http://www.math.utah.edu/$\sim$dwiggles/} %
\email{\tt drwiggle@uark.edu}
\thanks{\today}

\begin{document}

\begin{abstract} We prove the Farrell-Jones conjecture for
  free-by-cyclic groups.  The proof uses recently developed
  geometric methods for establishing the Farrell-Jones Conjecture.
\end{abstract}

\thanks{ The first author is supported by the NSF under the grant
  number DMS-1607236. 
  The second author is
    supported in part by Grant-in-Aid for Scientific Research (No.15H05739).  The third author
  would like to thank the Fields Institute for its hospitality. }
\subjclass[2010]{Primary 20F65, 18F25.}
\maketitle

\section{Introduction}
\label{sec:intro}

The Farrell-Jones Conjecture (which we will frequently abbreviate
by FJC) was formulated in \cite{FJ:FJC}. For a group $G$ relative to the family
$\vcyc$ of virtually cyclic subgroups it predicts an isomorphism between
the $K$-groups (resp.\ $L$-groups) of a group ring $RG$ and the
evaluation of a homology theory on a certain type of classifying space
for $G$. Such a computation, at least in principle, gives a way of
classifying all closed topological manifolds homotopy equivalent to a given
closed manifold of dimension $\geq 5$, as long as the FJC
is known for
the fundamental group. Particular significant consequences include the
Borel conjecture (a homotopy equivalence between aspherical manifolds
can be deformed to a homeomorphism) and the vanishing of the Whitehead
group $Wh(G)$ for torsion-free $G$.

The Farrell-Jones Conjecture has generated much interest in the last
decade, due in no small part to the recent development of an axiomatic
formulation that both satisfies useful inheritance properties, and
provides a method for proving the conjecture.  For example, FJC is
currently known for hyperbolic groups \cite{BLR:EquivariantCovers},
relatively hyperbolic groups \cite{Bar:RelHypFJC}, CAT$(0)$ groups
\cite{Wegner:FCJ-cat0}, virtually solvable groups
\cite{Wegner:FJC-Virt-solvable}, $\text{GL}_n(\mathbb{Z})$
\cite{BLRR:FJC-GLn}, lattices in connected Lie groups \cite{KLR}, and
mapping class groups \cite{BB:FJCMCG}. The reader is invited to
consult the papers
\cite{BLR:FJC-Applications,Bar:RelHypFJC,Bar:OnProofsofFJC,BB:FJCMCG}
for more information on applications of FJC and methods of proof.

The following theorems are the primary aims of the present note:
\begin{theorem}\label{thm:main}
  Let $\Phi\colon F_n\to F_n$ be an automorphism of a free group of
  rank $n$.  Then the mapping torus $G_\Phi=F_n\rtimes_\Phi\mathbb{Z}$
  satisfies the $K$- and $L$- theoretic Farrell-Jones Conjectures.
\end{theorem}
\begin{theorem}\label{thm:extensions}
  Let  
  \begin{equation*}
    1\longrightarrow F_n\longrightarrow G\longrightarrow Q\longrightarrow 1.
  \end{equation*}
  be a short exact
  sequence of groups.
  If $Q$ satisfies the $K$-theoretic (resp.\ $L$-theoretic)
  Farrell-Jones conjecture, then so does $G$.
\end{theorem}

The structure of this paper is as follows: Section
\ref{sec:background} contains the background material on the
Farrell-Jones Conjecture necessary for our purposes.  In particular,
we review the geometric group theoretic means for proving FJC
developed in \cite{FJ:FJC} and subsequently refined and generalized in
\cite{BLR:EquivariantCovers,BLR:FJC-Hyperbolic}.  We prove Theorem
\ref{thm:main} in Section \ref{sec:proof}, using train tracks for free
group automorphisms.  We then use a standard argument to push this
result to the case of free-by-FJC extensions in Section
\ref{sec:extensions}.

\section{Background}
\label{sec:background}

While the Farrell-Jones Conjecture originated in \cite{FJ:FJC}, its
present form with coefficients in an additive category is due to
Bartels-Reich \cite{BR:FJC-Coeffs} (in the $K$-theory case) and
Bartels-L\"uck \cite{BL:FJC-Twisted} (in the $L$-theory case).  The
existence of an axiomatic formulation of geometric conditions implying
the Farrell-Jones conjecture for a group $G$ allows us to feign
ignorance about the precise definitions of most of the objects in its
statement, as well as the statement itself. We therefore restrict
ourselves to the axiomatic formulation.

A \emph{family} $\mathcal{F}$ of subgroups of $G$ is a non-empty
collection of subgroups which is closed under conjugation and taking
subgroups.  For example, the collection
$\fin$ of finite subgroups of $G$ is a family, as is the
collection $\vcyc$ of virtually cyclic subgroups of $G$.

\subsection{Geometric Axiomatization of FJC}
\label{Sec:GeometricMethods}

We first recall a regularity condition that has been useful in recent
results on the Farrell-Jones conjecture
\cite{BLR:FJC-Hyperbolic,Bar:RelHypFJC,Knopf:FJCTrees}.  Let $X$ be a
space on which $G$ acts by homeomorphisms and $\mathcal{F}$ be a
family of subgroups of $G$.  An open subset $U\subseteq X$ is said to
be an $\mathcal{F}$-subset if there is $F\in\mathcal{F}$ such that $gU
= U$ for $g\in F$ and $gU\cap U = \emptyset$ if $g\notin F$.  An open cover
$\mathcal{U}$ of $X$ is \emph{$G$-invariant} if
$gU\in\mathcal{U}$ for all $g\in G$ and all $U\in\mathcal{U}$.  A
$G$-invariant cover $\mathcal{U}$ of $X$ is said to be an
\emph{$\mathcal{F}$-cover} if the members of $\mathcal{U}$ are all
$\mathcal{F}$-subsets.  The \emph{order} (or multiplicity) of a cover
$\mathcal{U}$ of $X$ is less than or equal to $N$ if each $x\in X$ is
contained in at most $N + 1$ members of $\mathcal{U}$.

\begin{definition}
  Let $G$ be a group and $\mathcal{F}$ be a family of subgroups.  An
  action $G\curvearrowright X$ is said to be
  \emph{$N$-$\mathcal{F}$-amenable} if for any finite subset $S$ of
  $G$ there exists an open $\mathcal{F}$-cover $\mathcal{U}$ of
  $G\times X$ (equipped with the diagonal $G$-action) with the
  following properties:
  \begin{itemize}
  \item the multiplicity of $\mathcal{U}$ is at most $N$;
  \item for all $x\in X$ there is $U \in \mathcal{U}$ with
    $S\times \{x\}\subseteq U$.
  \end{itemize}
  An action that is $N$-$\mathcal{F}$-amenable for some $N$ is said to
  be \emph{finitely $\mathcal{F}$-amenable}.  We remark that such
  covers have been called wide in some of the literature.
\end{definition}

\subsection{The class $\AC(\vnil)$}
\label{sec:inheritance}
Following \cite{BB:FJCMCG} we now define the class of groups
$\AC(\vnil)$ that satisfy suitable inheritance properties and all
satisfy FJC.  Let $\vnil$ denote the class of finitely generated
nilpotent groups and their subgroups. Set $\ac^0(\vnil)=\vnil$ and
inductively $\ac^{n+1}(\vnil)$ consists of groups $G$ that admit a
finitely $\mathcal{F}$-amenable action on a compact Euclidean retract
(ER) with all groups in $\mathcal{F}$ belonging to $\ac^n(\vnil)$. The
action on a point shows that $\ac^n(\vnil)\subseteq \ac^{n+1}(\vnil)$
and we set $\AC(\vnil)=\bigcup_{n=0}^\infty \ac^n(\vnil)$.

\begin{proposition}[\cite{BB:FJCMCG}]\label{prop:FJCInheritance}
    \begin{enumerate}
    \item [(i)] $\AC(\vnil)$ is closed under taking subgroups,
    taking finite index overgroups, and finite products.
    \item [(ii)] All
      groups in $\AC(\vnil)$ satisfy the Farrell-Jones Conjecture.
      \end{enumerate}
\end{proposition}

Our main result can now be stated as follows.

\begin{theorem}\label{main}
  Let $\Phi:F_n\to F_n$ be an automorphism. Then $G=F_n\ltimes_\Phi
  \Z$ belongs to $\AC(\vnil)$.
\end{theorem}

The following two theorems of Knopf and of Bartels will be crucial.

\begin{theorem}[{\cite[Corollary 4.2]{Knopf:FJCTrees}}]\label{Knopf}
  Let $G$ act acylindrically on a simplicial tree $T$ with finitely
  many orbits of edges. If all vertex
  stabilizers belong to $\AC(\vnil)$ then so does $G$.
\end{theorem}

\begin{theorem}[\cite{Bar:RelHypFJC}]\label{Bartels}
  Suppose $G$ is hyperbolic relative to a finite collection of
  subgroups, each of which is in $\AC(\vnil)$. Then $G$ also belongs
  to $\AC(\vnil)$.
\end{theorem}

\subsection{Background on outer automorphisms}
\label{sec:OuterAutomorphisms}
Outer automorphisms of a finitely generated, rank $n$ free group $F_n$
come in two flavors according to the growth rate of conjugacy classes
of $F_n$.  Indeed, let $\phi\in\Out(F_n)$ and let $\mathcal{C}$ denote
the set of conjugacy classes of elements of $F_n$, which comes with a
natural action of $\Out(F_n)$.  After choosing a basis
$\{x_1,\ldots,x_n\}$ for $F_n$, elements of $\mathcal{C}$ are in
correspondence with cyclically reduced words in the $x_i$'s.  We use
$|\cdot|$ to denote the word length of such a cyclically reduced
representative.  A standard fact in the field (that follows from the
existence of relative train tracks \cite{BH:TrainTracks}) is that each
$c\in\mathcal{C}$ has a well defined asymptotic exponential growth
rate: for all $c\in\mathcal{C}$, there exists $\lambda_{\phi,c}\geq 1$
such that
$\lim_{k\to\infty}\frac{1}{k}\log |\phi^k(c)|=\log\lambda_{\phi,c}$.
Moreover, as $c$ varies over all conjugacy classes in $F_n$, with $\phi$
fixed, the set of $\lambda$'s occurring is finite:
$\#\{\lambda_{\phi,c}\}_{c\in\mathcal{C}}<\infty$.  We say that
\emph{$\phi$ has exponential growth} if there exists $c\in\mathcal{C}$
whose corresponding $\lambda_{\phi,c}$ is greater than 1.  Otherwise,
we say \emph{$\phi$ has polynomial growth}; the justification for this
terminology requires machinery that we will not need.

The following structure theorem of Gautero-Lustig for free-by-cyclic groups is
crucial. The original proof remains incomplete as it relies on
unproven facts about certain type of train tracks. In the meantime
Ghosh and Dahmani-Li gave different proofs.

\begin{theorem}[\cite{GL:RelHyp,Ghosh:RelHyp,
    DL:RelativeHyperbolicityFreeProducts}]\label{thm:rel-hyp}
  Suppose $\Phi\colon F_n\to F_n$ grows exponentially.  Then $G_\Phi$
  is hyperbolic relative to a finite collection of subgroups, each of
  the form $F\rtimes_\psi\mathbb{Z}$ for a finite rank free group $F$
  and a polynomially growing automorphism $\psi\colon F\to F$.
\end{theorem}

\subsection{Transverse coverings}
\label{sec:TransverseCover}
Consider an action $G\curvearrowright T$ on a simplicial tree by
isometries.  A \emph{transverse covering} \cite[Definition
4.6]{G:LimitGroups} in $T$ is a $G$-invariant family $\mathcal{Y}$ of
non-degenerate closed subtrees of $T$ such that any two distinct
subtrees in $\mathcal{Y}$ intersect in at most one point.  A
transverse covering $\mathcal{Y}$ gives rise to a new tree, $S$,
called the \emph{skeleton of $\mathcal{Y}$} as follows: the vertex set
$V(S)=V_0\cup V_1$ where the elements of $V_0$ are in one-to-one
correspondence with elements of $\mathcal{Y}$, and the elements of
$V_1$ are in correspondence with nonempty intersections (necessarily
consisting of a single point) of distinct elements of $\mathcal{Y}$.
Edges are determined by set containment: there is an edge from
$x\in V_1$ to $Y\in V_2$ if $x\in Y$.  The action of $G$ on
$\mathcal{Y}$ determines an action $G\curvearrowright S$.

\section{Proof of Theorem \ref{main}}
\label{sec:proof}
This section contains a proof of our first main result, which follows
with relative ease by combining results from the literature with
Proposition \ref{prop:PGFJC} below.  After explaining the reduction to the case of polynomially growing automorphisms,
the remainder of the section is devoted to the proof of Proposition \ref{prop:PGFJC}.

\begin{proposition}\label{prop:PGFJC}
  Suppose $\phi\in\Out(F_n)$ has polynomial growth and let $\Phi$ be
  any lift to $\Aut(F_n)$.  Then the free-by-cyclic group $G_\Phi$
  belongs to $\AC(\vnil)$.
\end{proposition}

It's easily seen that the group $G_\Phi=F_n\rtimes_{\Phi}\mathbb{Z}$
depends only on the outer class of $\Phi$, so we will henceforth
denote it by $G_\phi$.
\begin{proof}[Proof of Theorem \ref{thm:main}]
  If $\phi\in Out(F_n)$ has polynomial growth, then
  Proposition \ref{prop:PGFJC} applies and we are done.  Otherwise,
  $\phi$ has exponential growth and we can apply
  Theorem \ref{thm:rel-hyp}
  to conclude that $G_\phi$ is hyperbolic relative to a collection of
  peripheral subgroups each of which is isomorphic to the mapping
  torus of a polynomially growing free group automorphism. By Theorem
  \ref{Bartels} the conclusion follows.
\end{proof}

\subsection{Proof of Proposition \ref{prop:PGFJC}}
Since $\AC(\vnil)$ passes to finite index overgroups, we may replace
$\phi$ by a positive power. This allows us to
apply \cite[Theorem 5.1.5]{BFH:Tits1} to obtain an improved
  relative train track map representing (a power of) $\phi$: i.e., a homotopy
  equivalence $f\colon \Gamma\to \Gamma$ of a finite marked graph
  satisfying many useful properties.  The properties needed for the
  current argument are as follows: there is an ordering
  $E_1,\ldots,E_m$ of the edges of $\Gamma$ such that:
  \begin{enumerate}
  \item For every $i$, the graph $\Gamma_i=E_1\cup\cdots\cup E_i$ has
    no valence one vertices.
  \item \label{item:2} For every $i$, $f(E_i)=u_iE_iv_i$ with
    $u_i,v_i$ (possibly trivial) edge paths in $\Gamma_{i-1}$.
  \end{enumerate}
  We now prove that $G_\phi=\pi_1(M_f)$ belongs to $\AC(\vnil)$ by
  induction on the number, $m$, of edges in the marked graph $G$.  If
  $m=1$, then $\phi$ is the identity automorphism.
  In this case, $G_\phi\simeq \mathbb{Z}^2$, which belongs to $\AC(\vnil)$.

  For the inductive step, we use property (\ref{item:2}) to modify $f$
  by a homotopy and arrange that $f$ fixes an interval $J$ in the
  interior of $E_m$ and that $f^{-1}(J)=J$.  Let
  $M_f=\Gamma\times[0,1]/(x,0)\sim (f(x),1)$ be the mapping torus.  In
  $M_f$, the interval $J$ gives rise to an embedded annulus,
  $J\times S^1$, and hence (by Van Kampen's Theorem) to a splitting of
  $\pi_1(M_f)=G_\Phi$ over an infinite cyclic group: either
  $G_\Phi=H\ast_\ZZ H'$ or $G_\Phi=H\ast_\ZZ$ according to whether or
  not $E_m$ is separating in $\Gamma=\Gamma_m$.  The properties of $f$
  stated above imply that the groups $H$ and $H'$ are mapping tori of
  free group automorphisms represented by graphs satisfying the above
  properties and which have fewer edges than $\Gamma$, so they belong
  to $\AC(\vnil)$ by induction.

  \begin{figure}[h]
    \begin{center}
     \def\svgwidth{\linewidth}
     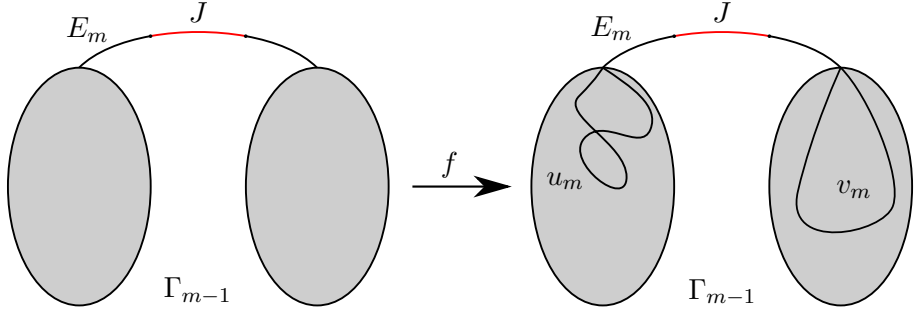
     \caption{The interval $J$ corresponds to an embedded annulus in
       $M_f$.}
      \label{fig:JSplitting}
    \end{center}
  \end{figure}
  
  Let $T$ be the Bass-Serre tree associated to this splitting.  It is
  not necessarily the case that $G_\phi\curvearrowright T$ is
  acylindrical.  In fact, if $E$ and $E'$ are two edges in $T$, then
  \begin{claim}
    Either $\Stab(E)=\Stab(E')$ or $\Stab(E)\cap\Stab(E')=1$.
  \end{claim}
  \begin{proof}[Proof of Claim]
    Consider the standard presentation,
    \begin{equation*}
      G_\Phi=\langle x_1,\ldots,x_n,t\mid tx_it^{-1}=\Phi(x_i)\rangle.
    \end{equation*}
    We remind the reader that every element of $h\in G_\Phi$ can be
    written uniquely as $gt^k$ for some integer $k$ and some element
    $g\in F_n$ as follows: if
    $h=g_0t^{k_1}g_1t^{k_2}\ldots t^{k_l}g_l$, then using the
    relations we have
    \begin{equation*}
      h=g_0\Phi^{k_1}(g_1)\Phi^{k_1+k_2}(g_2)\ldots\Phi^{\sum k_i}(g_l)t^{\sum k_i}
    \end{equation*}

    Now the fundamental group of the embedded annulus,
    $\pi_1(J\times S^1)$, is identified in $G_\Phi$ with the stable
    letter, $t$.  Evidently the edge stabilizers in $T$ are precisely
    the conjugates of $\langle t\rangle$.  Using the observation
    above, we see that any conjugate of $t$ can be written in the form
    $gt$ for some $g\in F_n$ and moreover that the element $g$ is
    unique.  Thus, if $\Stab(E)\cap\Stab(E')\neq 1$, then
    $\langle gt\rangle\cap \langle t\rangle\neq 1$ and therefore
    $g=1\in F_n$, so $\Stab(E)=\Stab(E')$.
  \end{proof}

  Returning to the proof, for an edge $E$ of $T$, we let $T_E$ be the
  forest in $T$ consisting of edges whose stabilizer is equal to
  $\Stab(E)$.  That $T_E$ is connected follows from the fact that $T$
  is a tree.  We are interested in the stabilizer $\Stab(T_E)$;
  suppose that $\Stab(E)=\langle t\rangle$ and that $w\in \Stab(T_E)$
  is such that $w^{-1}\cdot E=E'\in T_E$.  As above, we can write $w$
  uniquely as $w=gt^k$.  On one hand, the definition of $T_E$ provides
  that the stabilizer of $E'$ is equal to $\langle t\rangle$.  On the
  other hand,
  \begin{equation*}
    \Stab(E')=\langle wtw^{-1} \rangle
    = \langle(gt^k) t(t^{-k}g^{-1})\rangle
    =\langle gtg^{-1}\rangle
    =\langle g\Phi(g)^{-1}t\rangle.
  \end{equation*}
  We conclude that
  $\Stab(T_E)\simeq \langle t\rangle\times \Fix(\Phi)$, and recall
  that $\Fix(\Phi)$ is a finite rank free group
  \cite{Ger:Finite-rank,BH:TrainTracks}. Since $\AC(\vnil)$ is closed
  under taking products, and free groups belong to $\AC(\vnil)$, we
  conclude that $\Stab(T_E)$ belongs to $\AC(\vnil)$.

  The subtrees $\{T_E\}_{E\in T}$ form a transverse covering of $T$
  (this is implied by the Claim).  Letting $S$ denote the skeleton of
  this transverse cover (refer to \S\ref{sec:TransverseCover} for
  definitions and notation), we observe the following:
  \begin{claim}
    The action $G_\Phi\curvearrowright S$ is acylindrical.
  \end{claim}
  \begin{proof}[Proof of Claim]
    Let $v,v'\in V(S)$ be vertices with $d(v,v')\geq 6$ and suppose
    $g\in\Stab(v)\cap\Stab(v')$.  By moving to adjacent vertices if
    necessary, we may assume that $v,v'\in V_1$ so that they are
    labeled by intersections of subtrees in $\mathcal{Y}$ (i.e.,
    points in $T$) rather than by trees themselves.  We will again
    denote by $v$ and $v'$ the corresponding points in $T$.  Now $g$
    must fix the segment in $T$ connecting $v$ to $v'$, and after
    moving to adjacent vertices we still have $d_S(v,v')\geq 4$.  In
    particular, there are two vertices in $V_0$ on the segment in $S$
    between $v$ and $v'$; hence the segment connecting $v$ and $v'$ in
    $T$ contains edges in two distinct subtrees $T_E$ and $T_{E'}$.
    So $g\in \Stab(T_E)\cap\Stab(T_{E'})$ must stabilize two edges
    with different stabilizers and so $g=1$ by the first claim.
  \end{proof}
  Returning one last time to the proof of the proposition we remark
  that the vertex stabilizers in $S$ come in two flavors: stabilizers
  of vertices in $V_1$ are subgroups of vertex stabilizers in $T$,
  which belong to $\AC(\vnil)$ by induction on the number $m$ of edges; and
  stabilizers of vertices in $V_0$, which are isomorphic to
  $F\times \mathbb{Z}$ and also belong to $\AC(\vnil)$.
  Finally, applying Theorem \ref{Knopf} proves Proposition
  \ref{prop:PGFJC}.


\section{Extensions}
\label{sec:extensions}
In this final section, we use the transitivity principle to deduce
Theorem \ref{thm:extensions} from the results of the preceding one.

\newtheorem*{thm:extensions}{Theorem \ref{thm:extensions}}
\begin{thm:extensions}
  Let 
  \begin{equation*}
    1\longrightarrow F_n\longrightarrow
    G\overset{\pi}{\longrightarrow} Q{\longrightarrow} 1.
  \end{equation*}
   be a short exact
  sequence of groups
  If $Q$ satisfies the $K$-theoretic (resp.\ $L$-theoretic)
  Farrell-Jones conjecture, then so does $G$.
\end{thm:extensions}

Given a homomorphism $\varphi\colon G\to H$ and a family of subgroups
$\mathcal{F}$ of $H$, the family of subgroups
$\{K\leq G\mid \phi(K)\in\mathcal{F}\}$ is denoted by
$\varphi^*\mathcal{F}$.  When $\varphi$ is the inclusion of a
subgroup, we sometimes denote $\varphi^*\mathcal{F}$ by
$\mathcal{F}|_G$.  We now recall the two relevant results from
\cite{BLR:FJC-Applications}, which we have rephrased to suit our
purposes:

\begin{lemma}[{\cite[Lemma 2.3]{BLR:FJC-Applications}}]
  \label{lem:pullbacks}
  Let $\pi\colon G\to Q$ be a group homomorphism and let $\mathcal{F}$
  be a family of subgroups.  If $Q$ satisfies the Farrell-Jones
  Conjecture (for either $K$-theory or $L$-theory) relative to
  $\mathcal{F})$, then $G$ satisfies the Farrell-Jones Conjecture
  relative to $\pi^*(\mathcal{F})$.
\end{lemma}

\begin{theorem}[{\cite[Theorem 2.4]{BLR:FJC-Applications}}]
  \label{thm:transitivity}
  Let $\mathcal{F}\subseteq \mathcal{G}$ be two families of subgroups
  of $G$.  Assume that for every element $H\in\mathcal{G}$, the group
  $H$ satisfies the Farrell-Jones Conjecture for $\mathcal{F}|_H$.
  Then $G$ satisfies the Farrell-Jones Conjecture relative to
  $\mathcal{F}$ if and only if $G$ satisfies FJC relative to
  $\mathcal{G}$.
\end{theorem}

\begin{proof}[Proof of Theorem \ref{thm:extensions}]
  We first apply Theorem \ref{thm:transitivity} taking
  $\mathcal{F}=\vcyc$ and $\mathcal{G}=\pi^*(\vcyc)$; the assumption
  on elements of $\mathcal{G}$ is equivalent to the statement of
  Theorem \ref{thm:main}.  We can therefore conclude that $G$
  satisfies FJC relative to $\vcyc$ if and only if $G$ satisfies FJC
  relative to $\pi^*(\vcyc)$.  Lemma \ref{lem:pullbacks} says that the
  latter of these two statements holds provided that $Q$ satisfies the
  Farrell-Jones Conjecture relative to $\vcyc$, which is evident.
\end{proof}

\bibliographystyle{alpha}%
\bibliography{bibliography}

\begin{thebibliography}{{Gho}18}

\bibitem[Bar16]{Bar:OnProofsofFJC}
Arthur Bartels.
\newblock On proofs of the {F}arrell-{J}ones conjecture.
\newblock In {\em Topology and geometric group theory}, volume 184 of {\em
  Springer Proc. Math. Stat.}, pages 1--31. Springer, [Cham], 2016.

\bibitem[Bar17]{Bar:RelHypFJC}
A.~Bartels.
\newblock Coarse flow spaces for relatively hyperbolic groups.
\newblock {\em Compos. Math.}, 153(4):745--779, 2017.

\bibitem[BB16]{BB:FJCMCG}
A.~{Bartels} and M.~{Bestvina}.
\newblock {The Farrell-Jones Conjecture for mapping class groups}.
\newblock {\em ArXiv e-prints}, June 2016.

\bibitem[BFH00]{BFH:Tits1}
Mladen Bestvina, Mark Feighn, and Michael Handel.
\newblock The {T}its alternative for {${\rm Out}(F_n)$}. {I}. {D}ynamics of
  exponentially-growing automorphisms.
\newblock {\em Ann. of Math. (2)}, 151(2):517--623, 2000.

\bibitem[BH92]{BH:TrainTracks}
Mladen Bestvina and Michael Handel.
\newblock Train tracks and automorphisms of free groups.
\newblock {\em Ann. of Math. (2)}, 135(1):1--51, 1992.

\bibitem[BL10]{BL:FJC-Twisted}
Arthur Bartels and Wolfgang L\"uck.
\newblock On crossed product rings with twisted involutions, their module
  categories and {$L$}-theory.
\newblock In {\em Cohomology of groups and algebraic {$K$}-theory}, volume~12
  of {\em Adv. Lect. Math. (ALM)}, pages 1--54. Int. Press, Somerville, MA,
  2010.

\bibitem[BLR08a]{BLR:EquivariantCovers}
Arthur Bartels, Wolfgang L\"{u}ck, and Holger Reich.
\newblock Equivariant covers for hyperbolic groups.
\newblock {\em Geom. Topol.}, 12(3):1799--1882, 2008.

\bibitem[BLR08b]{BLR:FJC-Hyperbolic}
Arthur Bartels, Wolfgang L\"uck, and Holger Reich.
\newblock The {$K$}-theoretic {F}arrell-{J}ones conjecture for hyperbolic
  groups.
\newblock {\em Invent. Math.}, 172(1):29--70, 2008.

\bibitem[BLR08c]{BLR:FJC-Applications}
Arthur Bartels, Wolfgang L\"uck, and Holger Reich.
\newblock On the {F}arrell-{J}ones conjecture and its applications.
\newblock {\em J. Topol.}, 1(1):57--86, 2008.

\bibitem[BLRR14]{BLRR:FJC-GLn}
Arthur Bartels, Wolfgang L\"{u}ck, Holger Reich, and Henrik R\"{u}ping.
\newblock K- and {L}-theory of group rings over {$GL_n({\bf Z})$}.
\newblock {\em Publ. Math. Inst. Hautes \'{E}tudes Sci.}, 119:97--125, 2014.

\bibitem[BR07]{BR:FJC-Coeffs}
Arthur Bartels and Holger Reich.
\newblock Coefficients for the {F}arrell-{J}ones conjecture.
\newblock {\em Adv. Math.}, 209(1):337--362, 2007.

\bibitem[DL19]{DL:RelativeHyperbolicityFreeProducts}
Fran{\c{c}}ois {Dahmani} and Ruoyu {Li}.
\newblock {Relative hyperbolicity for automorphisms of free products}.
\newblock {\em arXiv e-prints}, page arXiv:1901.06760, Jan 2019.

\bibitem[FJ93]{FJ:FJC}
F.~T. Farrell and L.~E. Jones.
\newblock Isomorphism conjectures in algebraic {$K$}-theory.
\newblock {\em J. Amer. Math. Soc.}, 6(2):249--297, 1993.

\bibitem[Ger87]{Ger:Finite-rank}
S.~M. Gersten.
\newblock Fixed points of automorphisms of free groups.
\newblock {\em Adv. in Math.}, 64(1):51--85, 1987.

\bibitem[{Gho}18]{Ghosh:RelHyp}
Pritam {Ghosh}.
\newblock {Relative hyperbolicity of free-by-cyclic extensions}.
\newblock {\em arXiv e-prints}, page arXiv:1802.08570, Feb 2018.

\bibitem[GL07]{GL:RelHyp}
Francois Gautero and Martin Lustig.
\newblock {The mapping-torus of a free group automorphism is hyperbolic
  relative to the canonical subgroups of polynomial growth}.
\newblock {\em ArXiv e-prints}, July 2007.

\bibitem[Gui04]{G:LimitGroups}
Vincent Guirardel.
\newblock Limit groups and groups acting freely on {$\mathbb{R}^n$}-trees.
\newblock {\em Geom. Topol.}, 8:1427--1470 (electronic), 2004.

\bibitem[KLR16]{KLR}
Holger Kammeyer, Wolfgang L\"{u}ck, and Henrik R\"{u}ping.
\newblock The {F}arrell-{J}ones conjecture for arbitrary lattices in virtually
  connected {L}ie groups.
\newblock {\em Geom. Topol.}, 20(3):1275--1287, 2016.

\bibitem[{Kno}17]{Knopf:FJCTrees}
S.~{Knopf}.
\newblock {Acylindrical Actions on Trees and the Farrell-Jones Conjecture}.
\newblock {\em ArXiv e-prints}, April 2017.

\bibitem[Weg12]{Wegner:FCJ-cat0}
Christian Wegner.
\newblock The {$K$}-theoretic {F}arrell-{J}ones conjecture for {CAT}(0)-groups.
\newblock {\em Proc. Amer. Math. Soc.}, 140(3):779--793, 2012.

\bibitem[Weg15]{Wegner:FJC-Virt-solvable}
Christian Wegner.
\newblock The {F}arrell-{J}ones conjecture for virtually solvable groups.
\newblock {\em J. Topol.}, 8(4):975--1016, 2015.

\end{thebibliography}
\end{document}